\documentclass{article}

\usepackage{amsmath, amsthm, amsfonts, amssymb, cite}
\usepackage{wrapfig}
\usepackage{stackrel}
\usepackage{color}
\usepackage[hyperfootnotes=false]{hyperref}

\newtheorem{theorem}{Theorem}
\newtheorem{lemma}[theorem]{Lemma}
\newtheorem{prop}[theorem]{Proposition}

\newcommand{\lt}{\left}
\newcommand{\rt}{\right}
\newcommand{\bpm}{\begin{pmatrix}}
\newcommand{\epm}{\end{pmatrix}}
\newcommand{\bsm}{\lt(\begin{smallmatrix}}
\newcommand{\esm}{\end{smallmatrix}\rt)}
\newcommand{\beq}{\begin{equation}}
\newcommand{\eeq}{\end{equation}}

\renewcommand{\d}{\mathrm{d}}

\newcommand{\Z}{\ensuremath{\mathbb{Z}}}

\newcommand{\Q}{\ensuremath{\mathbb{Q}}}
\newcommand{\R}{\ensuremath{\mathbb{R}}}
\newcommand{\C}{\ensuremath{\mathbb{C}}}

\renewcommand{\H}{\ensuremath{\mathbb{H}}}

\renewcommand{\a}{\mathfrak{a}}

\newcommand{\vep}{\varepsilon}

\DeclareMathOperator{\GL}{GL}
\DeclareMathOperator{\SL}{SL}

\newcommand{\aquad}{\qquad\qquad}

\renewcommand{\th}{\textsuperscript{th}}
\newcommand{\inv}{^{-1}}
\newcommand{\hf}{\frac{1}{2}}
\newcommand{\qtr}{\frac{1}{4}}

\renewcommand{\Im}{\operatorname{Im}}

\newcommand{\kron}[2]{\lt(\frac{#1}{#2}\rt)}

\usepackage[usenames,dvipsnames]{xcolor}

\newcommand{\tf}{\tilde{f}}
\newcommand{\tA}{\tilde{A}}

\newcommand{\tW}{\tilde{W}}
\newcommand{\tF}{\tilde{F}}
\newcommand{\ttau}{\tilde{\tau}}
\renewcommand{\k}{\kappa}
\newcommand{\Vol}{\operatorname{Vol}}
\newcommand{\V}{\mathcal{V}}
\newcommand{\B}{\mathcal{B}}

\title{Bounds on Sup-Norms of Half Integral Weight Modular Forms}

\newcommand\blfootnote[1]{
\begingroup
\renewcommand\thefootnote{}\footnote{#1}
\addtocounter{footnote}{-1}
\endgroup
}

\begin{document}
\author{Eren Mehmet K{\i}ral}
\date{}
\maketitle

\begin{abstract}
	Bounding sup-norms of modular forms in terms of the level has been the focus of much recent study, \cite{blomer2010bounding}, \cite{TemplierSup3}. In this work the sup norm of a half integral weight cusp form is bounded in terms of the level, we prove that $\|y^{\frac{\k}{2}}\tf\|_\infty \ll_{\vep,\k} N^{\hf -\frac{1}{18}+\vep}\|y^{\frac{\k}{2}}\tf\|_{L^2}$ for a modular form $\tf$ of level $4N$ and weight $\k$, a half integer.\blfootnote{2010 \emph{Mathematics Subject Classification}: Primary 11F03; Secondary 11F37.}
\blfootnote{\emph{Key words and phrases}: $L^\infty$, Sup-Norm, half-integral weight, hyperbolic lattice count, amplification.}
\end{abstract}

\section{Introduction}\label{sec:introduction}

	Modular forms are basic objects, which have proved invaluable to the theory of $L$-functions and analytic number theory. In its simplest form, they can be regarded as functions on the upper half plane $\H$ which satisfy prescribed transformation formulas under a discrete group of isometries $\Gamma$. In this work we will bound the sup-norm of $y^{\frac \k2}|\tf(z)|$ in terms of its $L^2$ norm, where $\tf$ is a half integer weight modular form of weight $\k$ with respect to the arithmetic group $\Gamma = \Gamma_0(4N)$. This yields an understanding how $\tf$ behaves as a function. We are interested in inequalities of the following form
	\[
		\|y^{\frac\k2}\tf\|_{\infty} \ll_{\vep} N^{\alpha + \vep}\|y^{\frac \k2}\tf\|_{L^2}.
	\]
	If one can prove this inequality for a smaller $\alpha$, this can be interpreted as $\tf$ being more equally distributed. If we normalize the volume of $\Gamma \backslash \H$ to one, then we cannot expect any value lower than $\alpha = 0$ for the above inequality, because the bulk of the function has to be stored somewhere.
	
	In the theory of automorphic forms, the study of sup-norms has started with \cite{iwaniec1995norms}, where Iwaniec and Sarnak bounded the sup-norm of a Maass form $f$ with respect to a power of its Laplacian eigenvalue. In recent work, see \cite{jorgenson2013uniform}, \cite{blomer2010bounding}, \cite{TemplierSup3}, several authors have bounded the sup norm of integral weight modular forms with respect to their to their level. Jorgenson and Kramer used heat kernel methods in \cite{jorgenson2004bounding} to compare the Arakelov metric with the hyperbolic metric. As a corollary they were able to obtain the bound $\|yf\|_\infty \ll \|yf\|_{L^2}N^{\hf + \vep}$ for weight 2 holomorphic modular forms $f$ of trivial nebentypus. Later in \cite{jorgenson2013uniform} this was generalized to arbitrary even weight. Their result is in fact very general and applies to hyperbolic surfaces with finite volume and their covers of arbitrary degree. Because of its generality we can think of the $N^{\hf+\vep}$ bound as a geometric one, and try to achieve better results by restricting ourselves to the case of a hyperbolic surface obtained from a congruence subgroup $\Gamma_0(4N)$. The $N^{\hf + \vep}$ bound can be easily seen to hold in the case when $\tf$ is of half integral weight and where $N$ is squarefree and odd. We show this in Theorem \ref{thm:trivialBound} using the Fourier expansions of $\tf$.
	
	For arithmetic surfaces $\Gamma_0(N)\backslash\H$ with squarefree $N$, Harcos and Templier have achieved the bound with $N^{\hf -\frac16 +\vep}$, for all integer weights with squarefree level $N$ and arbitrary nebentypus $\eta$ in \cite{TemplierSup3}. Furthermore, Templier in \cite{templier2012large} indicates that $\alpha = \qtr$ is the optimal value in general. He gives a family of weight $2$ cuspidal newforms with the level square of a prime.

	Let $\kappa$ be a half-integer, and $\tf$ a cuspidal modular form of level $4N$, weight $\k$ and nebentypus $\eta$. The function $y^{\k/2} |\tf(z)|$ is bounded on the upper half plane. Furthermore assume that $\tf$ is an eigenfunction of all the half-integer weight Hecke operators. These definitions will be further detailed in Section \ref{sec:notation}. We prove the following theorem comparing the sup norm and the $L^2$-norms of a modular form of half integer weight. The proof also applies to integer weight modular forms.
	
	\begin{theorem}\label{thm:supBound}
		Let $\tf$ be a modular, cuspidal Hecke eigenform of weight $\k$, level $4N$, nebentypus $\eta$ where $\k$ is a half integer and $N$ is odd, squarefree.  One has the following inequality,
		\[
			\|y^{\k/2}\tf\|_\infty \ll_{\vep,\k}  \|y^{\k/2}\tf\|_{L^2}N^{\hf - \frac{1}{18}+\vep}.
		\] 
		Here $\vep$ is any positive real quantity.
	\end{theorem}

	In proving Theorem \ref{thm:supBound} we will closely follow the method in \cite{TemplierSup2} using spectral expansions of half-integer weight automorphic kernels and the amplification method of Duke-Friedlander-Iwaniec. The new difficulty in the half-integral weight case that we overcome is the fact that in this context the Hecke operators are only supported on the square integers. This results in the weaker bound of $N^{\hf-\frac{1}{18}+\vep}$ rather than $N^{\hf-\frac{1}{6} + \vep}$, see the sketch of proof below. Secondly, the Atkin-Lehner theory of modular forms is different, with no involution at the prime $2$. We overcome this difficulty by bounding simultaneously a whole basis of Hecke eigenfunctions as well as $\tf|_\k[W(2)]$ and $\tf|_\k[\bsm 1&0\\2N&1\esm]$. This is done with the same method one uses to bound $\tf(z)$. We then show in Theorem \ref{thm:atkinLehner} that one may assume $z$ to lie in a fundamental set for the group $A_0(2N)$ which generated by the group $\Gamma_0(2N)$ and the corresponding Atkin-Lehner involutions.
	
	The above argument could also be applied to the proof in \cite{TemplierSup3}. We would like to remark that this slightly extends the domain of validity of the $\|fy^{\frac k2}\|_{L^\infty} \ll \|fy^{\frac k2}\|_{L^2} N^{\hf - \frac 16 + \vep}$ bound also to automorphic forms $f$ of weight $k$ level $4N$ with $N$ squarefree and odd. 
	
	 After the writing of this paper the referee has brought to my attention the preprint of Abhishek Saha, \cite{saha2014sup} where he bounds the sup-norm of level $N$ Maass forms by $N^{\hf - \frac{1}{12} + \vep}$. The novelty of this result over \cite{TemplierSup3} is that in this case $N$ does not have to be squarefree. In our case, also because of the $4$, we have a non-squarefree level. But since the square part of the level is fixed, we do not have to consider growth in the square part of the level. One common method, other than the general outline of \cite{TemplierSup2}, used in both \cite{saha2014sup} and in this paper is the consideration of bounding the modular form together with $f|[B]$ for some matrices $B \in \GL_2(\Q)$. Such a consideration is necessary since Atkin-Lehner operators are not enough to bring a modular form to a desired fundamental domain, and hence we bound all elements in $\mathcal{B}'$, defined in Theorem \ref{thm:atkinLehner}. At this point in order to bound $f|[B]$, Saha considers an automorphic kernel on a different group and a thinner set of Hecke operators, whereas we use the same kernel slashed in either variable as in \eqref{eq:slashedAutomorphicKernel} and different operators $T_{\k,\eta}^B$. 
	
	Here is a brief sketch of the proof. Let $K(z,w)$ be a half integral weight automorphic kernel as defined in Section \ref{subsec:automorphicKernel}. It is automorphic of weight $\k$ in the $z$ variable and of weight $-\k$ in the $w$ variable and has a spectral expansion as
	\[
		K(z,w) = \sum_{j} h(t_j) \tF_j(z)\overline{\tF_j(w)} + \text{ cts.}
	\]
	Here $\tF_j(z)$ are weight $\k$ Maass forms on the surface $\Gamma_0(4N)\backslash \H$. If you apply the Hecke operator $T_{\k,\eta}(\ell^2)$ on the $z$ variable, you obtain
	\[
		T_{\k,\eta}(\ell^2)K(z,w) = \sum_j h(t_j) \widetilde{\lambda}_j(\ell^2)\tF_j(z) \overline{\tF_j(w)} +\text{ cts.}
	\]
	After taking $z=w$, the quantity $|\tF_j(z)|^2$ appears on the right hand side. We consider a summation over $T_{\k,\eta}(\ell^2)K(z,w)$ with well chosen weights for each $\ell$, which \emph{amplifies} a single choice of $|\tF_j(z)|^2$. Then we ignore all other terms and a bound for the summation becomes a bound for $|\tF_j(z)|^2$. This is done in section \ref{subsec:amplification}.


\section{Notation}\label{sec:notation}

	For any complex number $z \in \C$, call
	\[
		e(z) = e^{2\pi i z}.
	\]
	
	Let $\H = \{z \in \C: \Im z >0\}$ be the upper half plane. If $z = x + iy$ is the variable used to denote an element of $\H$, the hyperbolic volume element in $\H$ is denoted by
	\[
		\d \mu (z) := \frac{\d x\d y}{y^2}.
	\]
	
	We will use the distance parameter
	\[
		u(z,w) = \frac{|z-w|^2}{4\Im(z) \Im(w)},
	\]
	which is related to the hyperbolic distance as follows,
	\[
		\cosh d_\H(z,w) = 2u(z,w) + 1.
	\]
	
	Let $\Gamma \leq \SL(2,\R)$ is a discrete group of isometries of $\H$ such that $\Gamma \backslash \H$ has finite hyperbolic volume. Define the Petersson inner product on automorphic forms of weight $k$ as
	\[
		\langle f,g \rangle = \frac{1}{\mathcal{V}} \iint_{\Gamma\backslash \H} f(z)\overline{g(z)} \d \mu(z)
	\]
	as long as the integral converges. The factor $\mathcal{V} := \Vol(\Gamma \backslash \H)$ is sometimes not included in the literature. The inclusion or omission of this factor changes the definition of the $L^2$-norm but not the sup-norm, so the result of Theorem \ref{thm:supBound} may differ by a factor of $N^{\hf}$ according to the chosen convention.
	 	
	For any $\k \in \R$ define the corresponding Laplacian operator
	\[
		\Delta_\k := -y^2\lt(\frac{\partial^2}{\partial x^2} + \frac{\partial^2 }{\partial y^2}\rt) + i\k y \frac{\partial}{\partial x}.
	\]
	We will only consider the cases $\k \in \hf \Z$.
	
	In \cite{shimura1973modular}, Shimura has defined a weight $\k$, level $4N$ modular cuspform $\tf$ with nebentypus $\eta$, as a holomorphic function on the upper half plane which satisfies the transformation formula
	\[
		\tf(\gamma z) =\eta(d) \vep_d^{-2\k} \kron{c}{d} (c z+ d)^\k \tf(z)
	\] 
	for all $\gamma = \bsm a&b\\c&d\esm \in \Gamma_0(4N)$, and vanishes at each cusp of $\Gamma_0(4N)$. Here
	\[
		\vep_d = 
		\begin{cases} 
		1 &\text{ if } d \equiv 1 \pmod4\\ 
		i &\text{ if } d \equiv 3 \pmod 4
		\end{cases}
	\]
	and $\kron{c}{d}$ is the Kronecker symbol extended as in the notation section of Shimura's paper \cite{shimura1973modular}.
	
	Shimura has also defined the notion of a Hecke operators $T_{\k,\eta}(\ell)$ on half integer weight modular forms, again as the action of double cosets via the slash operators. With respect to the Petersson inner product the adjoint of the Hecke operators are given by $T_{\k,\eta}(\ell)^* = \overline{\eta(\ell)}T_{k,\eta}(\ell)$ and this implies that the Hecke eigenvalues satisfy $\overline{\lambda(\ell)} = \overline{\eta(\ell)} \lambda(\ell) $ when $(\ell, 4N) = 1$. These operators also commute. One novelty of the theory in the half-integer weight case is that the operators are identically zero for all non-square $\ell$.    
	
	For $\k$ a half integer and $\eta$ a Dirichlet character modulo $4N$ denote the space of cusp forms of weight $\k$, level $4N$ and nebentypus $\eta$ by $S_\k(4N,\eta)$. Since the Hecke operators commute, one may choose a basis of simultaneous Hecke eigenfunctions. Given $\tf$ a modular form of half integral weight $\k$, we call $\tF(z) = y^{\frac\k2}\tf(z)$.
	 
	For $z \in \H$ and $\gamma \in \Gamma_0(4)$ call the half integer weight cocycle as $$J(\gamma, z) = \vep_d\inv \kron{c}{d} (cz+d)^\hf / |cz+d|^\hf,$$ where $\gamma = \bsm a&b\\c&d\esm$.	The function $\tF$ defined above, satisfies the transformation formula $\tF(\gamma z) = \eta(d) J(\gamma,z) \tF(z)$.
	
	Let $$\mathfrak{G} = \{(A,\phi(z)): A \in \GL_2(\R), \phi(z)^2 = t\det(A)^{-\hf}(cz + d)/|cz+d|\}.$$ Given  $\tF: \H \to \C$, define $\tF|_\k[(A,\phi)] = \phi(z)^{-2\k} \tF(Az)$. Furthermore if $\gamma \in \Gamma_0(4)$, we call, by abuse of notation, $\tF|_\k[\gamma] (z) = J(\gamma,z)^{-2\k}\tF(\gamma z)$. We shall also sometimes use the notation $\tF|_\k[A]$ for $A\in \GL_2(\R)$ to mean $\tF|_\k[(A,\phi)]$ if the choice of $\phi$ in $(A,\phi)\in\mathfrak{G}$ does not matter.


\section{Bound for the supremum norm} \label{sec:supNormBound}

	We first prove the following qualitative bound based on the Fourier expansion of $\tf$ at the $\infty$ cusp. This will help us to prove the geometric bound $N^{\hf+\vep}$. It will also be useful later on. 
	\begin{prop}\label{prop:fourierBound}
		Let $\tf$ be a half integral modular cuspform of weight $\k$ and level $4N$. Then for all $z = x + iy \in \H$ and for all $\vep >0$,
		\[
			y^{\frac \k2}|\tf(z)| \ll_{\k,\vep} \|y^{\frac \k2} \tf\|_{L^2}\frac{N^\vep}{\sqrt{y}}
		\]
	\end{prop}
	%
	\begin{proof}
		Given
		\[
			y^{\frac\k2} \tf(z) = \sum_{n=1}^\infty \tA(n) n^{\frac{\k-1}{2}}y^{\frac{\k}{2}}e^{2\pi in z},
		\]
		we split the sum into its head and tail. Put
		\[
			H = \sum_{n=1}^{N^\vep/y} \tA(n) n^{\frac{\k-1}{2}} y^{\frac \k2} e^{2\pi i nz} \qquad \text{and} \qquad T = \sum_{n\geq N^{\vep}/y} \tA(n) n^{\frac{\k-1}{2}}y^{\frac k2} e^{2\pi i nz}.
		\]
		for $\vep>0$. The tail $T$ can be bounded by the decay of the exponential. The Fourier coefficient is bounded as $\tA(n) = O(n^\hf)$ and therefore $T = O\lt(N^\vep e^{-N^{\vep}}(1 + 1/y)\rt)$, and
		\[
			H(X) \ll \lt(\sum_{n=1}^{N^\vep/y} |\tA(n)|^2 \rt)^\hf \lt(\sum_{n=1}^{N^\vep/y}  n^{\k-1} y^{\k}e^{-4\pi n y} \rt)^\hf. 
		\]
		The first  sum can be bounded by applying an inverse Mellin transform to the Dirichlet series
		\[
			\sum_{n=1}^\infty \frac{|\tA(n)|^2}{n^s} = \frac{(4\pi)^{s + \k-1}}{\Gamma(s + \k-1)} \langle E(z,s), |\tf(z)|^2y^\k \rangle
		\]
		where $E(z,s) = \sum_{\Gamma_\infty \backslash \Gamma_0(N)} \Im(\gamma z)^s$ is the weight zero Eisenstein series. The function $E(z,s)$, and hence the Dirichlet series, has a pole at $s = 1$, and the residue of the Dirichlet series equals $\frac{(4\pi)^\k}{\Gamma(\k)} \|y^{\frac{\k}{2}}\tf\|_{L^2}^2$. After taking an inverse Mellin transform, we can bound it with $O(N^\vep/y)$. The second factor can be majorized with the gamma integral, which converges. Combining, we have
		\[
			y^{\frac \k2} |\tf(z)| \ll_k \frac{N^\vep}{\sqrt{y}}\|y^{\frac \k2 }\tf\|_{L^2} + N^\vep e^{-N^\vep}\lt(1 + \frac 1y\rt)
		\]
		The second summand, coming from the tail of the sum, decays as $N\to \infty$, and rapidly so. Thus we get the result. 
	\end{proof}
	Notice that this proposition gives a better bound for $\tf$ near the cusp at infinity. Such behavior is expected as the terms of the Fourier expansion decay rapidly as one goes up the $\infty$ cusp. 
	
	Further note that in the above proof we only used the basic Fourier expansion of $\tf$. In what follows, we will use this theorem on $\tf|_\k[B]$ instead of $\tf$, where the matrix $B$ satisfies that $\Gamma' = \Gamma_0(4N) \cap B\inv\Gamma_0(4N)B$ is of finite index in $\Gamma_0(4N)$ and $B\inv \Gamma_0(4N)B$. Such $B$ are said to be in the commensurator of $\Gamma_0(4N)$.  We have the Fourier expansion
	\[
		\tf|_\k[B]= \sum_{n=1}^\infty \tilde{B}(n) n^{\frac{\k-1}{2}} e^{2\pi i n z}
	\] 
	and the Dirichlet series $\sum_{n=1}^\infty |\tilde{B}(n)|^2 n^{-s}$ is obtained from the inner product $\langle E(B\inv z,s,\Gamma'), y^\k|\tf(z)|^2 \rangle$ where we have taken the Eisenstein series with respect to the congruence subgroup $\Gamma'$ which has a pole at $s =1$ with constant residue. 
	
	The Atkin-Lehner involution matrices $W(Q)$ for $\Gamma_0(4N)$ for odd divisors $Q|N$, could be chosen as
	\begin{equation*}
		W(Q) =  \bpm Q^2\beta &4N/Q\\4N\gamma&Q\epm , 
	\end{equation*}
	where $ Q^2 \beta - \lt(\frac{4N}{Q}\rt)^2 \gamma = 1$. This is also an Atkin-Lehner operator for $\Gamma_0(2N)$; it is a determinant $Q$ matrix which is upper triangular modulo $2N$ and modulo $Q$ only the upper right corner is nonzero. The group $\Gamma_0(2N)$ has one extra Atkin-Lehner operator $W(2)$ which can be chosen to be of the form 
	\begin{equation*}
		W(2) = \bpm 2\alpha & \beta \\ 2N\gamma& 2\delta \epm 
	\end{equation*}
	so that $4\beta - N^2\gamma = 1$. The Atkin-Lehner operators normalize the group they belong to. In fact they form a large part of the normalizers of $\Gamma_0(2N)$, see \cite{bars2008group}.
	
	There are operators $\tilde{W}(Q)$ for odd $Q$ as defined by Ueda in \cite{ueda1993twisting} (p.151), these are slash operators of $\tW(Q) \in \mathfrak{G}$. One can choose $\operatorname{pr}(\tilde{W}(Q)) = W(Q)$, where $\operatorname{pr}: \mathfrak{G} \to \GL_2(\R)$ is projection onto the matrix component.

	Let $A_0(2N)$ be the group generated by the above $W(Q)$ (including $Q = 2$) and $\Gamma_0(2N)$. It is shown in \cite{TemplierSup2} (Lemma 2.2) that if the point $z\in \H$ is chosen so that it has the highest imaginary part among $\delta z$ as $\delta$ runs through $A_0(2N)$, then
	\begin{equation}\label{eqn:diophantine}
		\Im(z) \geq \frac{\sqrt{3}}{4N} \qquad \text{and}\qquad |cz+d|^2 \geq \frac1{2N}
	\end{equation}
	for any $(c,d) \in \Z^2 - (0,0)$. Call the set of such points $z$ as $\mathcal{F}(2N)$.
	
	\begin{theorem}\label{thm:atkinLehner}
		Let $N$ be squarefree and odd. Let $\mathcal{B}$ be the union over all characters $\eta$ of Hecke eigenbases for $S_\k(4N,\eta)$. Put $\mathcal{B'} = \B \cup \B|_\k[A] \cup \B|_\k[W(2)] \cup \B|_\k[AW(2)]$ where $A = \bpm 1&0\\2N&1\epm$. Then,
		\[
			\sup_{z\in \H} \max_{\tf \in\mathcal{B'}} y^{\frac \k2}|\tf(z)|
		\] 
		is attained at $z \in \mathcal{F}(2N)$.
	\end{theorem}

	\begin{proof}
		Let us use the notation of $\tF(z) = y^{\frac \k2}\tf(z)$. One sees from Ueda's and Kohnen's work that the operators $\tilde{W}(Q)$ send Hecke eigenforms in $S_\k(4N,\eta)$ to Hecke eigenforms in $S_\k(4N,\eta\kron{Q}{\cdot})$, see \cite{ueda1993twisting} Proposition 1.20 and \cite{kohnen1982newforms}. In fact $\tF$ and $\tF|_\k[\tW(Q)]$ have the same Hecke eigenvalues for Hecke operators $T_{\k,\eta}(\ell^2)$ and $T_{\k,\eta\kron{Q}{\cdot}}(\ell^2)$ respectively, with $(\ell^2, 4N)=1$.
		
		Now given $\tf \in \B'$, assume that $|\tF(z)|$ attains its maximum at a $w \in \H$. We can apply an element $\delta \in A_0(2N)$ so that $w = \delta z \in \mathcal{F}(2N)$. We may express the element $\delta$ as $W(Q) \gamma' W(2)^j$ where $Q$ is odd, $j \in \{0,1\}$ and $\gamma' \in \Gamma_0(2N)$. This is simply because the Atkin-Lehner operators normalize $\Gamma_0(2N)$. We may further decompose $\gamma' = \gamma A^i$ where $\gamma \in \Gamma_0(4N)$ and $A$ is an element chosen above so that it is in $\Gamma_0(2N)$ but not $\Gamma_0(4N)$, and $i \in \{0,1\}$. Since the latter is an index two subgroup in the former, $\Gamma_0(2N)$, and hence $A$, normalizes $\Gamma_0(4N)$. Finally we have $\delta = \gamma W(Q) A^i W(2)^j$.
 		\[
			|\tF(w)| = |\tF(\gamma W(Q)A^iW(2)^j z)|   = |\tF|_\k[\gamma W(Q)A^iW(2)^j](z)| = |\tF'(z)|
		\]
		where $\tF'(z) = y^{\frac \k2}\tf'(z)$ and $\tf'$ is another element of $\B'$.
	\end{proof}

	From this theorem, we see that given any $\tf$ it is enough to bound the function at points $z \in \mathcal{F}(2N)$, as long as the bound we have will be independent of a particular choice of $\tf$, but only depend on the level and weight. Hence Proposition \ref{prop:fourierBound} yields the following theorem.
	
	\begin{theorem}\label{thm:trivialBound}
		Let $N$ be squarefree and odd, and let $\tf$ be an element of $\B'$, as defined above. We have the upper bound
		\[
			y^{\frac \k2}|\tf(z)| \ll N^{\hf + \vep}\|y^{\frac \k2}\tf\|_{L^2}.
		\]
	\end{theorem}
	
	\begin{proof}
		By proposition \ref{prop:fourierBound} we have this bound for $y \gg 1/N$, for any $\tf \in \B'$. According to the remark after the Proposition, we can apply the theorem to $y^{\frac \k2}\tf|_{\k}[A^iW(2)^j](z)$ for $i,j \in \{0,1\}$. Then by equation \eqref{eqn:diophantine} one has \[y^{\frac \k2} |\tf(z)| \ll N^{\hf + \vep} \|y^{\frac \k2}\tf\|_{L^2}.\qedhere\]
	\end{proof}
		
	The idea of the proof of Theorem \ref{thm:supBound} is modeled on the proof in \cite{TemplierSup2}. We first consider $y^{\frac \k 2} \tf$ within a family of half integer weight Maass forms (eigenfunctions of the weight $\k$ Laplacian). A natural way of accomplishing this is to take an automorphic kernel $K(z,w)$ which is of weight $\k$ with respect to $z$ and of weight $-\k$ with respect to $w$, and expand it spectrally. With the specialization $w = z$ the value $y^{\k}|\tf(z)|^2$ will show up as a particular summand. The next goal is to amplify this summand so that bounding the whole sum would give us nontrivial results in bounding $y^{\frac{\k}{2}}|\tf(z)|$.
	
	\subsection{The Automorphic Kernel}\label{subsec:automorphicKernel}
	
		An automorphic kernel $K(z,w)$ on the surface $\Gamma_0(4N) \backslash \H$ with weight $\k$ a half-integer has been constructed in \cite{PattersonLaplacian1} as follows:
		
		There, Patterson extends the theory of automorphic kernels and point pair invariants of Selberg to the case of arbitrary real weight. Start with a positive even function $h(t)$ which can be extended to an analytic function in horizontal strips with sufficient decay. The exact conditions are given in p.91 of \cite{PattersonLaplacian1}. Using the Selberg transform one obtains a point pair invariant $k(z,w)$ in the upper half plane; that is, $k$ is a function of hyperbolic distance in the upper half plane. As a function of the distance between $z$ and $w$ equation (15) loc.\ cit.\ shows that $k(z,w)\ll u(z,w)^{-1 - \vep}$ (recall that $u(z,w)$ was defined in Section \ref{sec:notation}). The automorphic kernel is formed as the sum
		\[
			K(z,w) = \sum_{\gamma \in \Gamma_0(4N)} \eta(\gamma) J(\gamma,w)^{2\k} (( z, \gamma w))^\k k(z, \gamma w),
		\]
		where we have used the notation of loc.\ cit.\ to mean
		\[
			((z,w))^\k = \frac{(w-\bar{z})^{2\k}}{|w-\bar{z}|^{2\k}}.
		\]
		
		Furthermore, the quantity $((z,w))$ satisfies
		\[
			((\gamma z, \gamma w))^{\k} = J(\gamma,z)^{2\k}J(\gamma,w)^{-2\k} ((z,w))^{\k},
		\]
		where the $J$ function is the normalized half-integer weight cocyle defined in Section \ref{sec:notation}.
		
		The resulting function $K$ is automorphic in both variables, with weights $\k$ and $-\k$ and characters $\eta$ and $\bar{\eta}$ respectively:
		\begin{align*}
			K(\gamma z, w) &= \eta(\gamma)J(\gamma, z)^{2\k} K(z,w),\\
			K(z,\gamma w) &= \overline{\eta(\gamma)}J(\gamma, w)^{-2\k} K(z,w).
		\end{align*}
		
		We call $K(z,w)$ the automorphic kernel on the surface $\Gamma_0(4N) \backslash \H$, and as can be seen in \cite{PattersonLaplacian1} it satisfies the property that if $\tF$ is an eigenfunction of $\Delta_\k$ and automorphic of weight $\k$, then
		\[
			\langle K(z,*) , \overline{\tF}\rangle =\frac{1}{\mathcal{V}} \iint_{\Gamma\backslash \H} K(z,w) \tF(z)\d\mu(w) =\frac{1}{\mathcal{V}} h(t) \tF(z),
		\]
		where $t$ is the spectral parameter given by $\Delta_k \tF = (\qtr + t^2)\tF$. There is an inverse for the Selberg transform which allows one to obtain $h(t)$ from $k(z,w)$ given by
		\[
			h(t) \Im(z)^{\hf + it} = \iint_{\H} k(z,w) \Im(w)^{\hf + it} \d \mu(w).
		\]
		
		Thus, spectral expansion of $K(z,w)$ in $L^2(\Gamma\backslash\H, \k,\eta)\otimes L^2(\Gamma\backslash\H,-\k,\overline{\eta})$ is,
		\begin{align*}\label{eqn:spectralExpansion}
			K(z,w) &=\frac{1}{\V} \sum_{j} h(t_j) \tF_j(z) \overline{\tF_j(w)} \\
			&\aquad\qquad+\underbrace{\sum_{\mathfrak{a} \text{ cusp}  } \frac{1}{2\pi} \int\limits_0^\infty h(t) E_{\mathfrak{a},\k,\eta}^{(4N)} (z, \tfrac12+ it) \overline{ E_{\mathfrak{a},\k,\eta}^{(4N)} (w, \tfrac12+ it) } \d t}_\text{cts.}.
		\end{align*}
		Here $\tF_j$ are the weight $\k$ Maass forms and $E_{\mathfrak{a},\k,\eta}^{(4N)}(z,s)$ is the weight $\k$ Eisenstein series at the cusp $\a$. We have expanded the function $K(z,w)$ in a spectral basis and since the Laplacian commutes with Hecke operators we may further diagonalize so that $\tF_j$ are eigenvalues of all the Hecke operators $T_{\k,\eta}(\ell)$. It is stated in \cite{hoffstein1993eisenstein} that metaplectic Eisenstein series are Hecke eigenfunctions. The proof is similar to the integral weight case and is due to the fact that $y^s$ is an eigenfunction of the Hecke operators. 
		 
		From now on we will suppress the continuous part of the spectrum in the notation. Furthermore note that the functions of interest for us, half integral weight holomorphic cuspidal forms, show up in the initial sum, multiplied with a factor of $y^{\frac \k2}$. Notice that all the $\tF_j$'s are normalized to have $L^2$-norm equal to one. We will assume that $y^{\frac{\k}{2}} \tf$ under investigation also has been normalized.

	\subsection{Amplification} \label{subsec:amplification}
		
		Apply the $\ell$\th\ Hecke operator in the $z$ variable to both sides.
		\begin{equation}\label{eqn:heckeSpectral}
			T_{\k,\eta}(\ell) K(z,w)  = \frac{1}{\V}\sum_j \widetilde{\lambda}_j(\ell)h(t_j) \tF_j(z) \overline{\tF_j(w)} + \text{ cts.} 
		\end{equation}
		Actually only the square $\ell$'s make a contribution because in the half integer weight case, the $\ell$\th\ Hecke operator is 0 unless $\ell$ is a square. We normalize the Hecke eigenvalues $\widetilde{\lambda}_j(\ell) =: \ell^{\frac{\k-1}{2}}\ttau_j(\ell)$. Multiply \eqref{eqn:heckeSpectral} by some constants $y_\ell/\ell^{\frac{\k-1}{2}}$ and sum over $\ell$:
		\begin{equation}\label{eqn:sumHeckeOperatorsOnKernel}
			R:= \sum_{\ell =1}^\infty\frac{ y_\ell}{\ell^{\frac{\k-1}{2}}}  T_{\k,\eta}(\ell) K(z,w) = \frac{1}{\V}\sum_\ell y_\ell \lt(\sum_{j} \ttau_j(\ell) \tF_j(z) \overline{\tF_j(w)} + \text{ cts.}\rt).
		\end{equation}
		
		The left hand side can be evaluated explicitly, first off start with the double coset decomposition,
		\[
			\Gamma_0(4N) \lt( \begin{matrix} 1&0\\0&\ell\end{matrix} \rt)\Gamma_0(4N) = \bigcup_{\nu}\Gamma_0(4N) \xi_\nu.
		\]
		Now write
		\begin{align*}
			T_{\k,\eta}(\ell) K(z,w) &= \frac 1\ell \sum_{\nu} \lt.K(z,w)\rt|_\k [\xi_\nu] \\
			&= \frac{1}{\ell} \sum_{\nu} \ell^{\frac \k2} K(\xi_\nu z,w) J(\xi_\nu,z)^{-2\k} \overline{\eta(\xi_\nu)}\displaybreak[0]\\
			&= \ell^{\frac \k2 -1}\sum_{\nu}\overline{\eta(\xi_\nu)} J(\xi_\nu ,z)^{-2\k} \\
			&\aquad \times\sum_{\gamma \in \Gamma_0(4N)} \overline{\eta(\gamma)}k(\gamma \xi_\nu z,w) J(\gamma, \xi_\nu z)^{-2\k} ((\gamma \xi_\nu z,w))^\k \displaybreak[0] \\
			&= \ell^{\frac \k2-1} \sum_{\gamma \in \Gamma_0(4N) \lt(\begin{smallmatrix} 1&0\\ 0&\ell \end{smallmatrix}\rt)\Gamma_0(4N)} \overline{\eta(\gamma)}k(\gamma z, w) J(\gamma z, w)^{-2\k} ((\gamma z, w))^\k.
		\end{align*}
		
		Taking absolute values we obtain
		\begin{equation} \label{eq:spectralSumLessThan}
			R \leq \sum_{\ell=1}^\infty \frac{|y_\ell|}{\sqrt{\ell}} \sum_{\gamma \in M(\ell, 4N)} |k(\gamma z, w)|,
		\end{equation}
		where $M(\ell,4N)$ is the set of integral matrices with determinant $\ell$ and where the lower left entry is divisible by $4N$. After taking $w=z$ the quantity $y^{\k}|\tf(z)|^2$ will appear on the left hand side of this inequality alongside with other eigenfunctions of the laplacian $\Delta_\k$, amplified if we make a correct choice of $y_\ell$'s.
		
		Let us briefly mention how we will bound the functions $\tf(A^i W(2)^j z)$ for $i,j \in \{0,1\}$. Call $B = A^i W(2)^j$. First note that both $A$ and $W(2)$ normalize the group $\Gamma_0(4N)$ and also the double cosets consisting of determinant $\ell$ matrices with lower left entry divisible by $4N$. Consider the kernel,
		\begin{equation} \label{eq:slashedAutomorphicKernel}
			H(z,w) = K(z,w)|_{\k,z}[B]|_{\k,w}[B],
		\end{equation}
		i.e.\ we have applied the slash operator with respect to both $z$ and $w$. It does not matter which lift of $B$ we take in $\mathfrak{G}$. Instead of applying the Hecke operator in the $z$ variable, which is the sum of slashing by some coset representatives $\xi_\nu$, we use $B\inv \xi_\nu B$. Call this operator by $T_{\k,\eta}^B(\ell)$. 
		\[
			T_{\k,\eta}^B(\ell) H(z,w)  = \frac{1}{\V} \sum_j \tilde{\lambda}_j(\ell) \tF_j|_\k[B](z)\overline{\tF_j|_\k[B](w)} + \text{cts.}
		\]
		Notice that the Hecke eigenvalue are the same since
		\[
			T_{\k,\eta}^B (\ell) (\tF_j|_\k[B]) = (T_{\k,\eta}(\ell) \tF_j)|_\k[B].
		\]
		
		Define
		\[
			R^B:= \sum_{\ell =1}^\infty\frac{ y_\ell}{\ell^{\frac{\k-1}{2}}}  T_{\k,\eta}^B(\ell) H(z,w). 
		\]
		Then as before
		\begin{align*}
			T_{\k,\eta}^B H(z,w) &= \frac1\ell \sum_\nu H(z,w)|_\k[B\inv \xi_\nu B]\\
			& \leq \ell^{\frac \k2 -1} \sum_{\gamma \in \Gamma_0(4N)\bsm 1&0\\0&\ell \esm\Gamma_0(4N)} |k(\gamma Bz,Bw)|.
		\end{align*}
		We can majorize this last sum if we run $\gamma$ through a larger sum, as all the terms are positive. The larger sum is the double coset $\Gamma_0(2N) \bsm 1&\\&\ell\esm \Gamma_0(2N)$. Our choice of matrix $B$ normalizes this double coset, and therefore as before we get the inequality
		\begin{equation}\label{eqn:twistedSpectralSumLessThan}
			R^B \leq \sum_{\ell=1}^\infty \frac{y_\ell}{\sqrt{\ell}}\sum_{\gamma \in M(\ell,2N)}|k(\gamma z, w)|.
		\end{equation}
		The difference between $4N$ and $2N$ does not matter since we are only looking at the size of $M(\ell, N)$ asymptotically as $N \to \infty$.
		
		 We have shown that both $\tF$ and $\tF|_\k[B]$ can be bounded by the same quantity. After this excursion, we go back to the consideration of using this inequality to bound the value of $\tF(z)$.
		
		The spectral expression of $R$ (or $R^B$) as in \eqref{eqn:sumHeckeOperatorsOnKernel} can be considered as a sum over the discrete spectrum amplified at a given $\tF$, after some choice of $y_\ell$'s. Choose some complex numbers $x_\ell$ as follows. Let $\Lambda$ be a large real quantity, and let
		\[
			\mathcal{P}^2 := \{p^2 \text{ prime} : p \nmid N \text{ and } \Lambda \leq p\leq 2\Lambda\} \text{ and } \mathcal{P}^4 := \{ p^4 : p^2 \in \mathcal{P}^2\}.
		\]
		Complex numbers $x_\ell$ will be supported on such a set. Now for our fixed half integral weight cusp form $\tf$ we define the amplifier coefficients.
		\[
			x_\ell = \begin{cases} \operatorname{sgn}(\ttau_{\tf}(\ell)) &\text{ if } \ell \in \mathcal{P}^2 \cup \mathcal{P}^4 \\ 0 & \text{ otherwise}.\end{cases}
		\]
		Here by $\operatorname{sgn}(\ttau_{\tf}(\ell))$ we mean any complex value in $S^1$. We remark  that $\sqrt{\overline{\eta(\ell)}}\ttau_{\tf}(\ell)$ is real, meaning that the phase is only determined by $\eta(\ell)$. The idea is that $\overline{x_\ell} \ttau_{\tf}(\ell)$ is positive real for all $\ell$. For $\tF_j$ other than that coming from $\tf$, we expect that there will be considerable cancellation due to phase differences in the $\ell$ sum below.
	
		Consider the sum
		\begin{equation*}
			S\mspace{-4mu}:= \mspace{-4mu} \sum_{j} h(t_j) \lt| \sum_{\ell=1}^\infty \overline{x_\ell} \ttau_j(\ell)\rt|^2 \mspace{-7mu} |\tF_j(z)|^2 \mspace{-4mu} = \mspace{-7mu} \sum_{\ell_1,\ell_2=1} \mspace{-7mu} x_{\ell_1} \overline{x_{\ell_2}} \sum_j h(t_j) \overline{\ttau_j(\ell_1)} \ttau_j(\ell_2) |\tF_j(z)|^2 \mspace{-5mu}.
		\end{equation*}
		All the summands of $S$ are positive and therefore focusing on the case $\tF_j = y^{\frac k2}\tf$ we obtain,
		\[
			h(t_{\tf}) \lt| \sum_{\ell \in \mathcal{P}^2 \cup \mathcal{P}^4} |\ttau_{\tf}(\ell)|\rt|^2 y^\k |\tf(z)|^2 \leq S.
		\]
		
		The half integral weight Hecke relations are given in \cite{purkait2012hecke}. We translate them here as follows:
		\[
			\overline{\ttau_j(\ell_1)}\tau_j(\ell_2) 
			= 
			\begin{cases}
				\overline{\eta(\ell_1)} \ttau_j(\ell_1\ell_2) &\text{ if } (\ell_1,\ell_2) = 1\\
				\overline{\eta(\ell)} \ttau_j(\ell^2)  + 1 &\text{ if } \ell := \ell_1 = \ell_2 \in \mathcal{P}^2\\
				\overline{\eta(\ell)} \ttau_j(\ell^3) + \ttau_j(\ell) &\text{ if } \ell = \ell_1\in \mathcal{P} \text{ and } \ell_2 = \ell_1^2\\
				\overline{\eta(\ell^2)} \ttau_j(\ell^3) + \overline{\eta(\ell)}\ttau_j(\ell) &\text{ if } \ell = \ell_2  \in \mathcal{P}^2 \text{ and } \ell_1 = \ell_2^2\\
				\overline{\eta(\ell^2)} \ttau_j(\ell^4) + \overline{\eta(\ell)}\ttau_j(\ell^2) + 1 &\text{ if } \ell^2 = \ell_1 = \ell_2 \in \mathcal{P}^4. 
			\end{cases}
		\]
		Therefore, the products $\overline{\ttau_j(\ell_1)} \ttau_j(\ell_2)$ can be written as sums of $\ttau_j(\ell)$'s. Let us make this change and collect all the terms with the same index. Then we call $y_\ell$ to be the coefficient of $\ttau_j(\ell)$, which is the same for all $j$. Including the continuous part of the spectrum and using \eqref{eq:spectralSumLessThan}, with $w$ specialized to the point $z$, obtain the following:
		\begin{align}
			h(t_{\tf}) \lt| \sum_{\ell \in \mathcal{P}^2 \cup \mathcal{P}^4} |\ttau_{\tf}(\ell)|\rt|^2 y^\k|\tf(z)|^2 
			&\leq \sum_j h(t_{j}) \lt| \sum_{\ell \in \mathcal{P}^2 \cup \mathcal{P}^4} \overline{x_\ell} \ttau_{j}(\ell)\rt|^2 |\tf_j(z)|^2 + \text{cts.}\notag\\
			&= \sum_{\ell} y_\ell \sum_j h(t_j) \ttau_j(\ell) |\tf_j(z)|^2  + \text{ cts.}\notag\\
			&\leq \V \sum_{\ell=1}^\infty \frac{|y_\ell|}{\sqrt{\ell}} \sum_{\gamma \in M(\ell, N)} |k(\gamma z, z)|. \label{eq:sumOverMatrices}
		\end{align}

		Just as in  \cite{TemplierSup2} we have the formula $\ttau_{\tf}(p^2)^2 - \ttau_{\tf}(p^4) = \eta(\ell^2)$ which follows from Theorem 1 of \cite{purkait2012hecke}. This forces $\max\{|\ttau_{\tf}(\ell^2)|, |\ttau_{\tf}(\ell^4)|\}\geq \hf$ and therefore gives us 
		\begin{equation} \label{eq:amplifierLength}
			\lt| \sum_{\ell \in \mathcal{P}^2 \cup \mathcal{P}^4}|\ttau_{\tf}(\ell)| \rt| \gg \Lambda^{1-\vep}.
		\end{equation}
		The occurance of $\vep$ could be replaced by a $\log \Lambda$ in the denominator as it stems from the number of primes in the interval $[\Lambda, 2\Lambda]$.
				
		We summarize the discussion above in a proposition.
		\begin{prop}\label{prop:boundByMatrices}
			Let $\tilde{g}$ be a half integral weight cusp form of weight $\k$ and level $4N$. Now let $\tf = \tilde{g}$ or $\tf= \tilde{g}|_\k[B]$ for $B$ as above. Then with numbers $y_\ell$ chosen as above 
			\[
				y^\k |\tf(z)|^2 \ll \frac{\V}{\Lambda^{2-\vep}} \sum_{\ell=1}^\infty \frac{|y_\ell|}{\sqrt{\ell}}\sum_{\gamma \in M(\ell,N)} |k(\gamma z,z)|.
			\]
		\end{prop}
		
		Now we bound the right hand side of this equation. Note that $y_\ell = 0$ unless $\ell$ is a square. In fact we have that $y_1 \ll \Lambda$, $|y_\ell| \leq 2$ in case $\ell = \ell_1^2 \ell_2^2, \ell_1^4\ell_2^4, \ell_1^2\ell_2^4, \ell_1^2\ell_1^6,\ell_1^4 \ell_1^8$ where $\ell_1$ and $\ell_2$ are distinct primes in the interval $[\Lambda, 2\Lambda]$. It is $0$ otherwise.
		
		From here onwards we cite the theorems in \cite{TemplierSup3} to bound the right hand side of \eqref{eq:sumOverMatrices}.

	\subsection{Counting Matrices}\label{subsec:countingMatrices}
		Now let us count the matrices in the right hand side of \eqref{eq:sumOverMatrices}. For that purpose we cite Lemma 4.2 from \cite{TemplierSup2}. 
		\begin{lemma}[Harcos-Templier]\label{lem:harcosTemplier}
			Let $K = 1 + LNy^2$. Call the number of integral matrices $\gamma = \bsm a&b\\c&d\esm$ such that $\det (\gamma)\leq L$ is a square, $c \equiv 0 \mod N$ and $u(\gamma z, z)\leq N^\vep$ by $M(z,\ell,N)$. Assuming $z \in \mathcal{F}(2N)$, this quantity is uniformly bounded by 
			\[
				\ll K L^\hf N^\vep.
			\]
		\end{lemma}
		
		
		\begin{proof}[Proof of Theorem \ref{thm:supBound}]
			With the notation $M(z,\ell,N)$ given above we can rephrase the above lemma as,
			\[
				\sum_{\substack{1\leq \ell \leq L\\ \ell \text{ is a square}}} M(z,\ell, N) \ll (1+LNy^2)L^\hf N^\vep
			\]
	
			Now notice that the summation on the right hand side of \eqref{eq:sumOverMatrices} can be divided into the groups
			\begin{align*}
				\ell = 1, && \Lambda^2 \leq \ell= \ell_1^2 \leq 4\Lambda^2, && \Lambda^4\leq \ell = \ell_1^2\ell_2^2 \leq 16\Lambda^4, \\
				&&\Lambda^6 \leq \ell = \ell_1^4 \ell_2^2 \leq 64\Lambda^6, && \Lambda^8 \leq \ell = \ell_1^4\ell_2^4 \leq 256\Lambda^8.
			\end{align*}
			where $\ell_1, \ell_2$ are arbitrary primes between the quantities $\Lambda$ and $2\Lambda$, perhaps equal. Let us call these sets of integers as $L_i$ with $i = 0,2,4,6,8$ respectively. Thus for $i = 2,4,6,8$ we bound
			\begin{equation*}\label{eq:matricesPartialBound}
				\sum_{\ell \in L_i} \frac{|y_\ell|}{\sqrt{\ell}} \sum_{\gamma \in M(z,\ell, N)} |k(\gamma z, z)| \ll  (1+\Lambda^{i}Ny^2) N^\vep.
			\end{equation*}
			
			As for $i = 0$ we have a different bound,
			\[
				|y_1| \sum_{\gamma \in M(z,1, N)} |k(\gamma z, z)| \ll \Lambda (1 + Ny^2)N^{\vep}.
			\] 
			
			Combining these various bounds we obtain
			\[
				y^\k|\tf(z)|^2 \ll \V \frac{N^\vep}{\Lambda^{2-\vep}} \lt(\Lambda(1+Ny^2) + 4 (1+\Lambda^8 Ny^2) \rt)
			\]
			and furthermore if one assumes that $ y \leq N^{-\frac{8}{9}}$, taking $\Lambda = N^{\frac 19}$ we obtain the bound
				\[
					y^{\frac{\k}{2}}|\tf(z)| \ll N^{\hf-\frac 1{18}+\vep}.
				\]
			
			Notice from Proposition \ref{prop:fourierBound} that  for $y \geq N^{-\frac{8}{9}}$, $$y^{\frac \k2}|\tf(z)| \ll N^{\hf - \frac 1{18}+\vep}. $$ 
			
			Two bounds match up, and we have the correct bound overall. Thus Theorem \ref{thm:supBound} is proven.
		\end{proof}

\subsection*{Acknowledgements}
	I would like to thank Nicolas Templier for his comments on the manuscript. I also would like to thank the referee for a very thorough and useful report and for pointing out the recent preprint of Abhishek Saha, \cite{saha2014sup}.
		
\bibliographystyle{acm}
\bibliography{MehmetBib}
\end{document}